\documentclass[a4paper,10pt]{article}

\usepackage{amsmath,amsthm,amsfonts,amssymb,stmaryrd,mathrsfs}
\usepackage{latexsym}
\usepackage{fullpage}
\usepackage{a4,palatino,fancyhdr}
\usepackage{graphicx}
\usepackage{float}  
\usepackage[small, bf, margin=90pt, tableposition=bottom]{caption}
\RequirePackage[T1]{fontenc}
\usepackage[breaklinks, naturalnames]{hyperref}
\usepackage[backrefs]{amsrefs}
\usepackage{enumerate}
\usepackage{multicol}
\usepackage{mathtools}
\usepackage[usenames,dvipsnames,svgnames,table]{xcolor}
\usepackage{mathbbol}

\input{xy}
\xyoption{all}
\xyoption{poly}
\usepackage[all]{xy}
\usepackage{colonequals}

\newtheorem{theorem}{Theorem}[section]
\newtheorem*{theorem*}{Theorem}
\newtheorem{proposition}[theorem]{Proposition}

\newtheorem{lemma}[theorem]{Lemma}
\newtheorem{corollary}[theorem]{Corollary}

\numberwithin{equation}{section}

\def\Tor{\mathop{\rm Tor}\nolimits}
\def\Ext{\mathop{\rm Ext}\nolimits}

\def\abel#1{#1^{\mathsf{ab}}}

\newcommand\ZZ{{\mathbb{Z}}}
\newcommand\NN{{\mathbb{N}}}
\newcommand\KK{\mathbb{k}}

\title{Isomorphisms of non noetherian down-up algebras}
\author{Sergio Chouhy and Andrea Solotar
\thanks{{\footnotesize This work has been supported by the projects  
UBACYT 20020130100533BA, UBACYT 20020130200169BA, PIP-CONICET 11220150100483CO, and MATHAMSUD-REPHOMOL.
The second author is a research member of CONICET (Argentina).}}
}
\date{}

\begin{document}

\maketitle

\begin{abstract}
 {
We solve the isomorphism problem for non noetherian down-up algebras $A(\alpha,0,\gamma)$ 
by lifting isomorphisms between some of their non commutative quotients. The 
quotients we consider are either quantum polynomial algebras in two variables for 
$\gamma=0$ or quantum versions of the Weyl algebra $A_1$ for non zero $\gamma$.
In particular we obtain that no other down-up algebra is isomorphic to the monomial 
algebra $A(0,0,0)$. We prove in the second part of the article that this is the only 
monomial algebra within the family of down-up algebras. Our method uses homological 
invariants that determine the shape of the possible quivers and we apply the 
abelianization functor to complete the proof.
}
\end{abstract}




\section{Introduction}
Let $\KK$ be a fixed field of characteristic $0$. Given parameters 
$(\alpha,\beta,\gamma)\in\KK^3$, the associated 
\textit{down-up algebra} $A(\alpha,\beta,\gamma)$, first defined in \cite{BR98}, is the 
quotient of the free associative algebra $\KK\langle d,u\rangle$ by the ideal generated 
by the relations
\begin{align*}
 &d^2u - (\alpha dud + \beta ud^2 + \gamma d),\\
 &du^2 - (\alpha udu + \beta u^2d + \gamma u).
\end{align*}
There are several well-known examples of down-up algebras such as $A(2,-1,0)$, isomorphic 
to the enveloping algebra of the Heisenberg-Lie algebra of dimension $3$, and, for 
$\gamma\neq0, A(2,-1,\gamma)$, isomorphic to the enveloping algebra of 
$\mathfrak{sl}(2,\KK)$.

A down-up algebra has a PBW basis given by
\begin{align*}
 \{u^i(du)^jd^k: i,j,k\geq0\}.
\end{align*}
Note that $A(\alpha,\beta,\gamma)$ can be regarded as a $\ZZ$-graded 
algebra where the degrees of $u$ and $d$ are respectively $1$ and $-1$. The field $\KK$ 
is the trivial module over $A(\alpha,\beta,\gamma)$, with $d$ and $u$ acting as $0$.

E. Kirkman, I. Musson and D. Passman proved in \cite{KMP99} that $A(\alpha,\beta,\gamma)$ 
is noetherian if and only it is a domain, if and only if $\beta\neq0$.

The isomorphism problem for down-up algebras was posed in \cite{BR98} where the 
authors considered algebras $A(\alpha,\beta,\gamma)$ of four different 
types and proved, by studying one dimensional modules, that algebras of 
different types are not isomorphic. They considered the following types,
\begin{multicols}{2}
\begin{enumerate}[(a)]
 \item $\gamma=0$, $\alpha+\beta=1$,
 \item $\gamma=0$, $\alpha+\beta\neq1$,
 \item $\gamma\neq0$, $\alpha+\beta=1$,
 \item $\gamma\neq0$, $\alpha+\beta\neq1$.
\end{enumerate}
\end{multicols}
As a consequence, we can restrict the question of whether two 
down-up algebras are isomorphic to each of the four types. In 
\cite{CM00} the authors solved the isomorphism problem for noetherian down-up 
algebras of types (a), (b) and (c) for every algebraically closed field $\KK$, and 
also for noetherian algebras of type (d) when in addition $\mathsf{char}(\KK)=0$. More 
precisely, they proved the following result.

\begin{theorem*}[\cite{CM00}]
Let $A=A(\alpha,\beta,\gamma)$ and $A'=A(\alpha',\beta',\gamma')$ be noetherian down-up 
algebras. Then $A$ is isomorphic to $A'$ if and only if 
\begin{enumerate}
 \item $\gamma=\lambda\gamma'$ for some $\lambda\in\KK^\times$, and
 \item either $\alpha'=\alpha$, $\beta'=\beta$ or 
$\alpha'=-\alpha\beta^{-1}$, $\beta'=\beta^{-1}$.
\end{enumerate}
\end{theorem*}
Their solution focuses mainly on the possible commutative quotients of down-up algebras. 
In contrast with this, there are very well studied non commutative algebras that appear 
as quotients of non noetherian down-up algebras, for example, when 
$\alpha\in\KK^\times$, the quantum plane $\KK_\alpha[x,y]$ and the quantum Weyl algebra 
$A_\alpha^1$,
\begin{align*}
  &\KK_\alpha[x,y] \colonequals \KK\langle x,y\rangle/\left(yx - \alpha xy\right),
  & &&A_\alpha^1 \colonequals \KK\langle x,y \rangle/\left( yx - \alpha xy - 1\right).
\end{align*}

In this article we describe isomorphisms amongst non noetherian down-up algebras by 
using these quotients. Our main result is the following.
\begin{theorem}
\label{teo:isos_down_up}
 Let $\KK$ be an algebraically closed field and let $\alpha,\alpha',\gamma,\gamma'\in 
\KK$. The algebras $A(\alpha,0,\gamma)$ and $A(\alpha',0,\gamma')$ are isomorphic if and 
only if 
\begin{enumerate}
 \item $\gamma = \lambda\gamma'$, for some $\lambda\in \KK^\times$, and 
 \item $\alpha'=\alpha$.
\end{enumerate}
\end{theorem}

We obtain in particular that no other down-up algebra is isomorphic to $A(0,0,0)$. In 
Section \ref{sec:monomial_down-ups} we prove 
\begin{theorem}
\label{teo:no_son_mono}
The algebra $A(\alpha,\beta,\gamma)$ is monomial if and only if 
$(\alpha,\beta,\gamma)=(0,0,0)$.
\end{theorem}
So, the only monomial algebra in the family of down-up algebras is the evident one. Our 
starting point is the fact that noetherian down-up algebras cannot be monomial since they 
are a domain of global dimension $3$ \cite{KMP99}.
The situation can be related to $3$-dimensional Sklyanin algebras. In both cases, an 
algebra $A$ is noetherian if and only if it is a domain. For Sklyanin algebras, these 
conditions are equivalent to $A$ being monomial \cite{Sm12}. This is not the case for 
down-up algebras.

Our proof uses homological invariants that determine the possible shapes of the quiver.
We think that these methods may be useful for other families of algebras.

\section{Isomorphisms of non noetherian down-up algebras}

The purpose of this section is to prove Theorem \ref{teo:isos_down_up}. Let $\KK$ be an 
algebraically closed field.
Note that the condition $\gamma = \lambda\gamma'$ for $\lambda\in \KK^\times$ is 
equivalent to the condition of $\gamma$ and $\gamma'$ being both zero or both non zero. 
We already know from \cite{BR98} that if $\gamma\neq0$, then $A(\alpha,0,\gamma)$ is 
isomorphic to $A(\alpha,0,1)$. This is done by rescaling $d$ by $\gamma d$.
 Also, observe that $A(\alpha,0,0)$ is not isomorphic to $A(\alpha',0,1)$ for 
all $\alpha,\alpha'\in \KK$, since they belong to different types. Gathering all this 
information, we deduce that Theorem \ref{teo:isos_down_up} is equivalent 
to the following two propositions: 

\begin{proposition}
 \label{prop1}
Let $\alpha,\alpha'\in \KK$. The algebras $A(\alpha,0,0)$ and $A(\alpha',0,0)$ are 
isomorphic if and only if $\alpha=\alpha'$.
\end{proposition}

\begin{proposition}
\label{prop2}
Let $\alpha,\alpha'\in\KK$. The algebras $A(\alpha,0,1)$ and $A(\alpha',0,1)$ are 
isomorphic if and only if 
$\alpha=\alpha'$.
\end{proposition}

We will thus prove both of them in order to obtain our result.
 
\begin{lemma}
\label{lemma:cocientes_down_up}
 Let $\alpha\in\KK^\times,$ $\gamma\in \KK$ and let $A=A(\alpha,0,\gamma)$ be a 
down-up algebra. Denote $\omega\colonequals du - \alpha ud - \gamma$. 
The algebra $A/\langle\omega\rangle$ is isomorphic to $\KK_\alpha[x,y]$ if $\gamma=0$ and 
it is isomorphic to $A_\alpha^1$ if $\gamma\neq0$.
Moreover, in case $\gamma=0$ or $\gamma=1$, the isomorphism maps the 
class of $d$ to $y$ and the class of $u$ to $x$.
\end{lemma}
\begin{proof}
 The algebra $A(\alpha,0,\gamma)$ is the quotient of the free algebra generated by 
the variables $d$ and $u$ subject to the relations $d^2u - \alpha dud - \gamma d=0$ and 
$du^2 - \alpha udu - \gamma u=0$. Denote by $\Omega$ the element $du -\alpha ud 
-\gamma$ in the free algebra. The projection of $\Omega$ onto 
$A(\alpha,0,\gamma)$ is $\omega$. The defining relations of $A$ are 
$d\Omega=0$ and $\Omega u=0$. Therefore, the algebra $A/\langle \omega\rangle$ 
is isomorphic to the algebra freely generated by letters $d,u$ 
subject to the relation $\Omega=0$. If $\gamma=0$, then this is exactly the 
definition of $\KK_\alpha[x,y]$. If $\gamma\neq0$, then $\omega= 
\gamma^{-1}((\gamma d)u - \alpha u(\gamma d) - 1)$, and so $A/\langle 
\omega\rangle$ is the quantum Weyl algebra generated by $x$ and $y$, with $y = \gamma d$ 
and $x=u$.
\end{proof}

In \cite{RS06} the authors describe all isomorphisms and automorphisms for quantum 
Weyl algebras $A_\alpha^1$ for $\alpha\in \KK^\times$ not a root 
of unity. In \cite{SV15} this result is generalized to the family 
of \textit{quantum generalized Weyl algebras}, including the 
quantum plane and the quantum Weyl algebra for all values of $\alpha\in 
\KK^\times$. We recall some of their results in the cases relevant to us.

\begin{theorem}[\cite{RS06},\cite{SV15}]
\label{teo_mariano_quimey_andrea}
 Let $\alpha,\alpha'\in \KK\setminus\{0,1\}$. 
 \begin{enumerate}[i)]
 \item The algebras $\KK_\alpha[x,y]$ and $\KK_{\alpha'}[x,y]$ are isomorphic if and 
only if $\alpha'\in\{\alpha,\alpha^{-1}\}$. Moreover, if 
$\varphi:\KK_\alpha[x,y]\to \KK_{\alpha^{-1}}[x,y]$ is an isomorphism and 
$\alpha\neq-1$, then there exist $\lambda,\mu\in \KK^\times$ such that 
$\varphi(x) = \lambda y \,\text{ and }\varphi(y) = \mu x.$
\item The algebras $A_\alpha^1$ and 	$A_{\alpha'}^1$ are isomorphic if and only if 
$\alpha'\in\{\alpha,\alpha^{-1}\}$. If $\alpha\neq-1$, 
then every isomorphism $\eta:A_\alpha^1\to A_{\alpha^{-1}}^1$ is of the form
$\eta(x)= \lambda y\,\text{ and }\eta(y) = -\lambda^{-1}\alpha^{-1}x,$
for some $\lambda\in \KK^\times$.
\end{enumerate}
\end{theorem}
%

\begin{proof}[Proof of Proposition \ref{prop1}]
 Let $\alpha$ and $\alpha'$ be elements of $\KK$ and suppose there exists an isomorphism 
of $\KK$-algebras $\varphi: A(\alpha,0,0)\to A(\alpha',0,0)$. Denote 
$A\colonequals A(\alpha,0,0)$, 
$A'\colonequals A(\alpha',0,0)$ and let $d'$ and $u'$ be the usual generators of $A'$.
 

Suppose $\alpha,\alpha'\in \KK\setminus\{0,1\}$ and $\alpha\neq\alpha'$. Let 
$\omega\colonequals du - \alpha ud$ and $\omega'\colonequals d'u' - 
\alpha'u'd'$. By Lemma \ref{lemma:cocientes_down_up}, we can identify $A/\langle 
\omega\rangle$ with $\KK_{\alpha}[x,y]$, where the canonical projection $\pi:A\to 
\KK_{\alpha}[x,y]$ sends $d$ to $y$ and $u$ to $x$, and similarly for $A'/\langle 
\omega'\rangle$ and $\KK_{\alpha'}[x,y]$; here we denote by $\pi'$ the canonical 
projection. Define $\psi_1=\pi'\circ\varphi$.
The equalities $d\omega=0$ and $\omega u=0$ hold in $A$, so 
\[\psi_1(d)\psi_1(\omega)=0=\psi_1(\omega)\psi_1(u).\]
The algebra $\KK_{\alpha'}[x,y]$ is a non commutative domain generated by $\psi_1(d)$ and 
$\psi_1(u)$. Thus, $\psi_1(d)$ and $\psi_1(u)$ are not zero and from the above equations 
we deduce $\psi_1(\omega)=0$. This implies that there 
exists an algebra map $\overline\psi_1:\KK_{\alpha}[x,y] \to \KK_{\alpha'}[x,y]$ such 
that $\psi_1=\overline\psi_1\circ\pi$. In the other 
direction we obtain that $\psi_2\colonequals\pi\circ\varphi^{-1}$ factors as 
$\psi_2 = \overline\psi_2\circ\pi'$. Since 
$\overline\psi_1\circ\overline\psi_2\circ\pi'=\pi'$ and 
$\overline\psi_2\circ\overline\psi_1\circ\pi = \pi$, we deduce 
$\overline\psi_1$ is an isomorphism. The situation is illustrated by the 
following commutative diagram,
\[
\xymatrix{A\ar[d]_-{\pi}\ar[r]^-{\varphi}&A'\ar[d]^-{\pi'}\\
\KK_\alpha[x,y]\ar@<0.8ex>[r]^{\overline\psi_1} & \KK_{\alpha'}[x,y] 
\ar@<0.8ex>[l]^-{\overline\psi_2}}
\]
By Theorem \ref{teo_mariano_quimey_andrea} and our assumption that 
$\alpha\neq\alpha'$, we obtain $\alpha'=\alpha^{-1}$. Theorem 
\ref{teo_mariano_quimey_andrea} also says that there exist $\lambda,\mu\in 
\KK^\times$ and $z_1,z_2\in\langle\omega'\rangle$ such that $\varphi(u)= \lambda 
d'+z_1$ and $\varphi(d)=\mu u'+z_2$.
Note that $A'$ is graded considering the 
generators $d'$ and $u'$ in degree $1$. 
Since $\deg(\omega')=2$, it follows that $z_1$ and $z_2$ are either $0$ or sums of 
homogeneous elements of degree at least $2$ with respect to this grading.
On the other hand $d^2u - \alpha dud$ is $0$, 
and so 
\[0=\varphi(d)^2\varphi(u) - \alpha\varphi(d)\varphi(u)\varphi(d).\] 
In particular the degree $3$ component of the right hand side of this equality, that is 
$(u')^2d' - \alpha u'd'u'$, must be $0$. But the set $\{(u')^i(d'u')^j(d')^l: 
i,j,l\in\NN_0\}$ is a $\KK$-basis of $A'$, and this is a contradiction. 
 
In case $\alpha=0$, an argument similar to the above one shows that there is an 
epimorphism $\psi: A\to \KK_{\alpha'}[x,y]$. As a consequence the elements $\psi(d)$ and 
$\psi(u)$ 
generate $\KK_{\alpha'}[x,y]$. If $\alpha'\neq0$, then the algebra 
$\KK_{\alpha'}[x,y]$ is a domain and it is not commutative, thus it cannot be generated 
by one element. From the equality $0=d^2u$ we obtain that 
$0=\psi(d^2u)=\psi(d)^2\psi(u)$, implying $\psi(d)=0$ or $\psi(u)=0$. This is a 
contradiction and so $\alpha'=0$.

If $\alpha=1$, then $A$ belongs to type (a) and so does $A'$. This implies 
$\alpha'=1$, concluding the proof of the proposition.
\end{proof}

Now we turn our attention to Proposition \ref{prop2}.
Let $A=A(\alpha,0,1)$ for $\alpha\in \KK$. Recall that $\omega\colonequals du - 
\alpha ud - 1$. Using Lemma 2.2 in \cite{Zh99}, the set $\{u^i\omega^jd^l: i,j,l\geq 
0\}$ is a $\KK$-basis of $A$.

\begin{lemma}
\label{lemma:omega^n}
The set $\{u^i\omega^jd^l: i,l\geq0\text{ and }j\geq1\}$ is a $\KK$-linear basis 
of the two sided ideal $\langle \omega \rangle$, and, for each 
$n\in\NN$, the set $\{u^i\omega^jd^l: i,l\geq0\text{ and }j\geq n\}$ is a 
$\KK$-linear basis of $\langle \omega \rangle^n$. 
\end{lemma}
\begin{proof}
Every element of the form $u^i\omega^jd^l$ 
with $j\geq1$ belongs to $\langle \omega\rangle$, so it only remains to prove
that $\langle\omega\rangle$ is contained in the $\KK$-vector space with 
basis the set $\{u^i\omega^jd^l: i,l\geq0\text{ and }j\geq1\}$. Given $z\in\langle 
\omega\rangle$, write $z=\sum_{i,j,l}\lambda_{i,j,l}u^i\omega^jd^l$ with 
$i,j,l\geq0$ and $\lambda_{i,j,l}\in \KK$. By Lemma \ref{lemma:cocientes_down_up} 
we can identify $A/\langle\omega\rangle$ with $A^1_\alpha$, and the canonical 
projection $\pi:A\to A^1_\alpha$ sends $u$ to $x$ and $d$ to $y$. The set 
$\{x^iy^l: i,l\geq0\}$ is a basis of $A^1_\alpha$. From the equalities
$\sum_{i,l}\lambda_{i,0,l}x^iy^l=\pi(z)=0,$ we deduce $\lambda_{i,0,l}=0$ for all 
$i,l\geq0$.

Taking into account the description we now have of $\langle\omega\rangle$, we see that 
the elements of $\langle\omega\rangle^2$ are linear combinations of monomials of type
$u^i\omega^jd^lu^{i'}\omega^{j'}d^{l'}$, with $j,j'\geq1$. Similarly, the 
elements of $\langle\omega\rangle^n$ are linear combinations of $n$-fold 
products of the same type. Therefore, to prove the second claim, it is 
sufficient to show that for every $r,s\geq0$ there exist $\lambda_i\in \KK$ such that 
$\omega d^ru^s\omega = \sum_{i\geq2}\lambda_i\omega^i$. Indeed, there exist 
$\lambda_{i,j,l}\in\KK$ such that
\[
 d^ru^s = \sum_{i,j,l\geq0}\lambda_{i,j,l}u^i\omega^jd^l.
\]
So
\[
 \omega d^ru^s \omega 
= \sum_{i,j,l\geq0}\lambda_{i,j,l}\omega u^i\omega^jd^l\omega
=\sum_{j\geq0}\lambda_{0,j,0}\omega^{j+2}.
\]
The last equality follows from $d\omega =0$ and $\omega u=0$.
\end{proof}

\begin{corollary}
\label{coro:formulas}
The set $\{[u^i\omega d^l]: i,l\geq0\}$, where 
$[p]$ denotes the class of an element $p$ in 
$\langle\omega\rangle/\langle\omega\rangle^2$, is a $\KK$-linear basis of the 
$A$-bimodule $\langle\omega\rangle/\langle\omega\rangle^2$. Moreover, in case  
$\alpha\neq1$, the following equalities hold

\begin{align*}
 &[u^i\omega d^lu] = \frac{\alpha^l-1}{\alpha-1}[u^i\omega d^{l-1}],\\
 &[du^i\omega d^l] = \frac{\alpha^i-1}{\alpha-1}[u^{i-1}\omega d^l],
\end{align*}
where the terms on the right are considered to be zero for $l=0$ or $i=0$.
\end{corollary}
\begin{proof}
 The first claim is a direct consequence of Lemma \ref{lemma:omega^n}. To prove the 
first formula, we fix $i\geq0$ and proceed by induction on $l$, the case $l=0$ being 
trivial from the equalities $\omega u=0=d\omega$. On the other hand, since $\omega^2 = 
\omega(du-\alpha ud -1) = \omega du - 
\omega$, we obtain that $\omega du  = \omega^2 + \omega$. Similarly $du\omega = 
\omega^2 + \omega$. Therefore, $[u^i\omega du] = [u^i\omega]$. Now, for $l\geq2$

\begin{align*}
 [u^i\omega d^lu] = [u^i\omega d^{l-2}(\alpha dud + d)] &= \alpha [u^i\omega 
d^{l-1}ud] + [u^i\omega d^{l-1}]\\
&=\alpha\frac{\alpha^{l-1}-1}{\alpha-1}[u^i\omega d^{l-2}d] + [u^i\omega 
d^{l-1}]\\
&=\frac{\alpha^l-1}{\alpha-1}[u^i\omega d^{l-1}].
\end{align*}

The second formula can be proved analogously.
\end{proof}

\begin{proof}[Proof of Proposition \ref{prop2}]
Let $\alpha,\alpha'\in \KK$. Denote $A\colonequals A(\alpha,0,1)$, 
$A'\colonequals A(\alpha',0,1)$. Let $d'$, $u'$ be the generators of 
$A'$. Suppose there exists an isomorphism of $\KK$-algebras $\varphi:A\to A'$. Recall 
that $\omega=du - \alpha ud - 1$ and $\omega'= d'u' - \alpha'u'd' - 1$.

If $\alpha=1$, then $A$ belongs to type (a), and so does $A'$, hence $\alpha'=1$. Now 
suppose $\alpha=0$ and $\alpha'\neq0$. By Lemma 
\ref{lemma:cocientes_down_up} the algebra $A'/\langle \omega'\rangle$ is 
isomorphic to $A_{\alpha'}^1$ and, if $\pi'$ denotes the canonical projection, then  
$\pi'(u')=x$ and $\pi'(d')=y$. Let $\psi = \pi'\circ\varphi$. Since 
$d\omega=\omega u =0$, we have $\psi(d)\psi(\omega)=\psi(\omega)\psi(u)=0$. 
Note that $\psi(d)$ and $\psi(u)$ generate $A_{\alpha'}^1$, and therefore 
they cannot belong to $\KK$; in particular they cannot be zero. We 
deduce $0=\psi(\omega)=\psi(d)\psi(u)-1$. The algebra $A_{\alpha'}^1$ has a filtration 
whose associated graded algebra $\mathsf{Gr}(A_{\alpha'}^1)$ is $\KK_{\alpha'}[x,y]$. The 
equality $\psi(d)\psi(u)=1$ implies that $\mathsf{Gr}(A_{\alpha'}^1)$ is not a domain, 
which is a contradiction since $\alpha'\neq0$.


Suppose $\alpha,\alpha'\in \KK\setminus\{0,1\}$ and $\alpha\neq\alpha'$. By the 
same arguments as in the proof of Proposition \ref{prop1}, the map 
$\psi\colonequals\pi'\circ\varphi :A\to A_{\alpha'}^1$ induces an 
isomorphism of $\KK$-algebras $\overline\psi:A_{\alpha}^1\to A_{\alpha'}^1$. Theorem 
\ref{teo_mariano_quimey_andrea} implies $\alpha'=\alpha$ or 
$\alpha'=\alpha^{-1}$. Since we are assuming $\alpha\neq\alpha'$, we deduce 
$\alpha'=\alpha^{-1}$. Again, by Theorem \ref{teo_mariano_quimey_andrea} we 
obtain that 
there exist $\lambda\in \KK^\times$ and $z_1,z_2\in\langle \omega\rangle$ such 
that 

\begin{align*}
 \varphi^{-1}(d') &= -\lambda^{-1}\alpha u + z_1,\\
 \varphi^{-1}(u') &= \lambda d + z_2.
\end{align*}

By rescaling the variables $d,u$, we may assume $\lambda=1$.
The equality $(d')^2u' - \alpha^{-1}d'u'd' - d'=0$ implies 

\begin{align*}
 0 &= \varphi^{-1}(d')^2\varphi^{-1}(u') - 
\alpha^{-1}\varphi^{-1}(d')\varphi^{-1}(u')\varphi^{-1}(d') - 
\varphi^{-1}(d')\\
 &=\alpha^2u^2d - \alpha udu + \alpha u + \alpha^2 u^2 z_2 - \alpha u z_1 d 
- \alpha z_1 ud \\ &\hskip0.5cm+ udz_1 - \alpha uz_2u + z_1du - z_1 + z\\
&=-\alpha u\omega + \alpha (\alpha u^2 z_2 - uz_1d) + (udz_1-\alpha uz_2u) + 
z_1\omega + z \in\langle \omega\rangle,
\end{align*}

where $z$ denotes the sum of all terms in which at least two factors $z_1$ or $z_2$ are 
involved.
Note that $z_1\omega$ and $z$ belong to $\langle\omega\rangle^2$. Taking classes modulo 
$\langle\omega\rangle^2$,
\[
 \alpha [u\omega] + \alpha([uz_1d]-\alpha[u^2z_2])  = [udz_1] - \alpha[uz_2u].
\]
Write now $z_1=\sum_{i,l\geq0,j\geq1}\lambda_{i,j,l}u^i\omega^jd^l$ and $z_2= 
\sum_{i,l\geq0,j\geq1}\mu_{i,j,l}u^i\omega^jd^l$. Using the formulas of 
Corollary \ref{coro:formulas} we obtain
\begin{align*}
 \alpha[u\omega] + &\sum_{i,l\geq0}\alpha(\lambda_{i,1,l}[u^{i+1}\omega 
d^{l+1}] - \alpha\mu_{i,1,l}[u^{i+2}\omega d^l]) =\\ 
&=\sum_{i\geq1,l\geq0}\lambda_{i,1,l}\frac{\alpha^{i}-1}{\alpha-1}[u^i\omega 
d^l] - \sum_{i\geq0,l\geq1} 
\alpha\mu_{i,1,l}\frac{\alpha^{l}-1}{\alpha-1}[u^{i+1}\omega d^{l-1}].
\end{align*}

By Corollary \ref{coro:formulas}, the set $\{[u^i\omega d^l]: i,l\geq0\}$ is a 
$\KK$-linear basis of $\langle\omega\rangle/\langle\omega\rangle^2$. For 
$m\geq0$, define $\Lambda_m\colonequals \lambda_{m+1,1,m} - 
\alpha\mu_{m,1,m+1}$. Looking at the coefficient 
corresponding to the term $[u^{m+1}\omega d^m]$ in the last equation for each $m\geq0$, 
we deduce

\begin{align*}
  &\alpha = \Lambda_0,\\
  &\alpha\Lambda_{m-1} = \frac{\alpha^{m+1}-1}{\alpha-1}\Lambda_{m},\,\text{ 
for 
}m\geq1.
 \end{align*}
The fact that $\Lambda_0=\alpha\neq0$ implies, by an inductive argument, that 
$\Lambda_{m}\neq0$ for all $m\in\NN$. As a consequence, either 
$\lambda_{m+1,1,m}\neq0$ for infinitely many values of $m\in\NN$, or $\mu_{m,1,m+1}\neq0$ 
for infinitely many values of $m\in\NN$. This is a contradiction that comes from the 
assumption $\alpha\neq\alpha'$.
\end{proof}

\section{Monomial down-up algebras}
\label{sec:monomial_down-ups}
An algebra is monomial if it is isomorphic to an algebra of the form 
$\KK Q/I$, 
where $Q$ is a quiver with a finite number of vertices and $I$ is a two-sided 
ideal in $\KK Q$ generated by paths of length at least $2$. The algebra $A(0,0,0)$ is 
monomial and no other down-up algebra is isomorphic to it. However, other monomial 
down-up algebras may exist.
In this section we prove Theorem \ref{teo:no_son_mono}.
Before doing it we prove a series of preparatory lemmas.

We will make use of the abelianization functor defined on $\KK$-algebras as
\[
 A\mapsto \abel A \colonequals A/J_A,
\]
where $J_A$ is the two sided ideal in $A$ generated by the set 
$\{xy-yx:x,y\in A\}$. The canonical projection $\pi_A:A\to \abel A$ is a 
natural transformation from the identity to the abelianization functor.

In order to state the next lemma we need some previous definitions.
Given a quiver $Q$ with a finite number of vertices and $e,e'\in Q_0$, 
define $eQ_1e'\colonequals\{\alpha\in Q_1: 
\mathsf{t}(\alpha)=e, \mathsf{s}(\alpha)=e'\}$, 
where $\mathsf{t}$ and $\mathsf{s}$ are the usual target and source maps. Also,
denote by $B_e$ the $\KK$-algebra $\KK[X_\alpha: \alpha\in eQ_1e]$. That is, $B_e$ is the 
polynomial algebra in variables indexed by 
the elements of the set $eQ_1e$. In case $eQ_1e=\emptyset$ we set $B_e=\KK$. If $I$ is a 
two-sided ideal in $\KK Q$ generated by 
paths of length at least $2$, define $I_e$ to be the ideal in $B_e$ 
generated by the set 
\[
\bigcup_{n\geq2}\{X_{\alpha_n}\cdots X_{\alpha_1}:\alpha_n\cdots\alpha_1\in I, 
\alpha_i\in eQ_1e\}.\] 

\begin{lemma}
 \label{lemma:monomiales_techo}
 Let $Q$ be a quiver with a finite number of vertices and $I$ a two-sided 
ideal in $\KK Q$ generated by paths of length at least $2$. There is an isomorphism of 
$\KK$-algebras 
\[\abel{\left(\frac{\KK Q}{I}\right)}\cong \bigoplus_{e\in Q_0}\frac{B_e}{I_e}.\]
\end{lemma}
\begin{proof}
The classes $\overline e$ in $\abel{\KK Q}$ of the vertices $e$ in $Q_0$ are a complete 
set of central orthogonal idempotents and $\overline e(\abel{\KK Q})\overline e$ is 
isomorphic to $\abel{(e\KK Qe)}$, and thus isomorphic to $B_e$. As a consequence, there 
is an isomorphism $\theta:\abel{\KK Q}\to \bigoplus_{e\in Q_0}B_e$ such that 
$\theta(\pi(\alpha))=X_\alpha$ for all $\alpha\in ekQe$ and $e\in Q_0$. Let $g$ be the 
map $\abel f\circ\theta^{-1}$. The commutativity of the diagram
\[
 \xymatrix{
  \KK Q \ar[r]^-{\pi}\ar[d]^-{f}  & \abel{\KK Q} \ar[d]^-{\abel f}\ar[r]^-{\theta} & 
\bigoplus_{e\in Q_0}B_e\ar[dl]^-{g}\\
  \KK Q/I \ar[r] & \abel{(\KK Q/I)}
 }
\]
implies that $g$ is surjective and that its kernel is $(\theta\circ\pi)(I + 
J_{\KK Q})=(\theta\circ\pi)(I)=\oplus_{e\in Q_0}I_e$, and the lemma follows.

\end{proof}

The following lemma is a well known result for finite dimensional algebras replacing 
$\Tor$ by $\Ext$. Here we give the proof for completeness. 
\begin{lemma}
\label{lemma:monomiales_maxtor1}
 Let $B=\KK Q/I$ be a monomial algebra. For each $e\in Q_0$ denote by $T_e$ the 
simple $B$-module corresponding to $e$. If $e,e'\in Q_0$, then 
$\#eQ_1e'=\dim_\KK\Tor_1^B(T_e,T_{e'})$. Moreover, if $Q$ has only one vertex 
$e$, then the dimension of $\Tor_1^B(T_e,T_e)$ is 
\[
 \sup\{\dim_\KK\Tor_1^B(T_1,T_2):T_1,T_2 \text{ are 
one dimensional $B$-modules}\}.
\]
\end{lemma}
\begin{proof}
Let $e,e'\in Q_0$. The first terms of the minimal projective bimodule resolution of $B$ 
are
\[
 \cdots\rightarrow B\otimes_E\KK Q_1\otimes_EB\rightarrow B\otimes_EB\rightarrow 
B\rightarrow 0,
\]
where $E=\KK Q_0$.

Applying the functor $T_e\otimes_B(\,-\,)\otimes_BT_{e'}$ to this resolution we obtain 
the following complex
\[
 \cdots\rightarrow T_e\otimes_\KK\KK eQ_1e'\otimes_\KK T_{e'}\rightarrow 
T_e\otimes_\KK T_{e'}\rightarrow 0,
\]
whose homology is isomorphic to $\Tor_\bullet^B(T_e,T_{e'})$. The minimality of the 
resolution implies that every arrow in the above complex is zero. As a consequence, 
$\Tor_1^B(T_e,T_{e'})\cong T_e\otimes_\KK\KK eQ_1e'\otimes_\KK T_{e'}\cong e\KK Q_1e'$, 
from where we deduce that the dimension of $\Tor_1^B(T_e,T_{e'})$ is $\#eQ_1e'$.
As for the second assertion, the same argument shows that if $T_1,T_2$ are 
one dimensional $B$-modules, then the homology of the complex
\[
 \cdots\rightarrow T_1\otimes_\KK\KK Q_1\otimes_\KK T_2\rightarrow 
T_1\otimes_\KK T_2\rightarrow 0,
\]
is isomorphic to $\Tor_\bullet^B(T_1,T_2)$. It follows that 
$\dim_\KK \Tor_1^B(T_1,T_2) \leq\dim_\KK(\KK Q_1)=\#Q_1=\dim_\KK \Tor_1^B(T_e,T_e)$, 
where $e$ is the only vertex of $Q$.
\end{proof}

\begin{lemma}
\label{lemma:cuentas_tor1}
 Let $A=A(\alpha,0,\gamma)$ and let $T_1,T_2$ be one dimensional $A$-modules.
 \begin{enumerate}[(i)]
  \item If $\gamma\neq0$ and $\alpha=1$, then $\dim_\KK \Tor_1^A(T_1,T_2)=0$.
  \item If $\gamma\neq0$ and $\alpha\neq1$, then 
$\dim_\KK \Tor_1^A(T_1,T_2)\leq1$.
  \item If $\gamma=0$, then $\dim_\KK \Tor_1^A(T_1,T_2) \leq2$ and 
$\dim_\KK \Tor_1^A(\KK,\KK) =2$. Moreover, if $\alpha\neq1$ and $T_1\neq \KK$, then 
$\dim_\KK\Tor_1^A(T_1,T_1)=1$.
 \end{enumerate}
\end{lemma}
\begin{proof}
 Let $T_1,T_2$ be one dimensional $A$-modules with bases $\{v_1\}$ and 
$\{v_2\}$, respectively. Let $\delta_1,\delta_2,\mu_1,\mu_2\in \KK$ be such that 
$d\cdot v_i=\delta_iv_i$ and $u\cdot v_i=\mu_iv_i$ for $i=1,2$.
From the equalities $d^2u-\alpha dud -\gamma d =0 =du^2 - \alpha udu-\gamma u$ we deduce
\begin{align}
\label{ecuaciones}
\begin{split}
  &\delta_i((1-\alpha)\delta_i\mu_i - \gamma)=0,\\
  &\mu_i((1-\alpha)\delta_i\mu_i - \gamma)=0,
\end{split}
\end{align}
for $i=1,2$. Consider the following resolution of $A$ as $A$-bimodule \cite{CS15}
\begin{align*}
\label{resolucion}
     0 \rightarrow A \otimes_\KK \Omega \otimes_\KK A 
           \overset{d_{3}}{\rightarrow} A \otimes_\KK R \otimes_\KK A 
           \overset{d_{2}}{\rightarrow} A \otimes_\KK V \otimes_\KK A 
           \overset{d_{1}}{\rightarrow} A \otimes_\KK A 
           \rightarrow 0,
\end{align*}
where $V$, $R$ and $\Omega$ are the subspaces of the free algebra $\KK\langle d,u\rangle$ 
spanned, respectively, by the sets $\{d,u\}$, $\{d^2u,du^2\}$ and $\{d^2u^2\}$. The 
differentials are
\begin{equation*}
\label{eq:delta2}
\begin{split}
d_{1}(1\otimes v\otimes 1) &=  v \otimes 1 - 1 \otimes v, \hskip0.5cm \text{for all 
}v\in V,\\
d_{2}(1 \otimes d^{2} u \otimes 1) &= 1 \otimes d \otimes d u +  d \otimes d 
\otimes u 
+ d^{2} \otimes u \otimes 1 
\\
&- \alpha (1 \otimes d \otimes u d + d \otimes u \otimes d  + d u \otimes d 
\otimes 1)
\\
&- \beta (1\otimes u \otimes d^{2} + u \otimes d \otimes d  + u d \otimes d 
\otimes 1)
\\
&- \gamma \otimes d \otimes 1,
\end{split}
\end{equation*}
\begin{equation*}
\label{eq:delta2bis}
\begin{split}
d_{2}(1 \otimes d u^{2} \otimes 1) &= 1 \otimes d \otimes u^{2}  + d \otimes u 
\otimes u  
+ d u \otimes u \otimes 1 
\\
&- \alpha (1\otimes u \otimes d u  +  u \otimes d \otimes u  +  u d \otimes u 
\otimes 1)
\\
&- \beta (1\otimes u \otimes u d  + u \otimes u \otimes d  +  u^{2} \otimes d 
\otimes 1)
\\
&- \gamma  \otimes u \otimes 1,
\end{split}
\end{equation*}
and 
\begin{equation}
\label{eq:delta3}
\begin{split}
d_{3}(1 \otimes d^{2} u^{2} \otimes 1) &= d \otimes d u^{2} \otimes 1 + \beta 
\otimes d u^{2} \otimes d  
\\
&-1 \otimes d^{2} u \otimes u  - \beta u \otimes d^{2} u \otimes 1.
\end{split}
\end{equation} 

Applying the functor $T_1\otimes_A(\,-\,)\otimes_A T_2$ to this resolution of $A$ we 
obtain the following complex of $\KK$-vector spaces whose homology is isomorphic to 
$\Tor_\bullet^A(T_1,T_2)$,
\[
\xymatrix{
 0 \ar[r] & 
 \KK\ar[r]^-{f_2} &
 \KK^2 \ar[r]^-{f_1} &
 \KK^2 \ar[r]^-{f_0} & 
 \KK \ar[r] &
 0,
}
\]
where
\begin{align*}
 &f_0=\begin{pmatrix}\delta_2-\delta_1 & \mu_2-\mu_1 \end{pmatrix}\text{ and} 
 \\[1em]
 &f_1=\begin{pmatrix}
 (1-\alpha)\delta_1\mu_1 + \delta_2(\mu_1-\alpha\mu_2)-\gamma & \mu_1(\mu_1-\alpha\mu_2)\\
 \delta_2(\delta_2-\alpha\delta_1) & (1-\alpha)\delta_2\mu_2 + 
\mu_1(\delta_2-\alpha\delta_1)-\gamma
\end{pmatrix}.
\end{align*}
The claims of the lemma follow from these formulas and Equation \eqref{ecuaciones}.
\end{proof}

We now turn to the proof of Theorem \ref{teo:no_son_mono}. Suppose 
$(\alpha,\beta,\gamma)\neq(0,0,0)$. Denote as usual $A(\alpha,\beta,\gamma)$ by $A$. Let 
$B=\KK Q/I$ be a monomial algebra and suppose there exists an isomorphism of 
$\KK$-algebras $\varphi:A\to B$. Since every down-up algebra has global dimension $3$ 
\cite{KK00},
we deduce that $I\neq0$ and so $B$ is not a domain, thus we get $\beta=0$.

Suppose $\gamma=0$. In this case $\alpha\neq0$ since we are assuming 
$(\alpha,\beta,\gamma)\neq(0,0,0)$. Note that

\[\abel A=\KK[d,u]/\langle(1-\alpha)d^2u,(1-\alpha)du^2\rangle,\]
in particular it is connected and so is $\abel B$. By Lemma \ref{lemma:monomiales_techo}, 
the quiver $Q$ has only one vertex $e$. Moreover, by Lemmas 
\ref{lemma:monomiales_maxtor1} and \ref{lemma:cuentas_tor1} we deduce 

\[2=\dim_\KK(T_e,T_e)=\#Q_1,\] 
thus, $Q$ has exactly two arrows $a$ and $b$. By Lemma 
\ref{lemma:cocientes_down_up} the quantum plane $\KK_\alpha[x,y]$
is a quotient of $A$ and so there exists an epimorphism $\varphi:B\to \KK_\alpha[x,y]$. 
Since $\alpha\neq0$, the quantum plane $\KK_\alpha[x,y]$ is a domain. Given a non zero 
path $p$ in $I$, $\varphi(p)=0$, implying that either $\varphi(a)$ or 
$\varphi(b)$ is zero. As a consequence, the quantum plane is generated as algebra by 
one variable, which is a contradiction.
%


Now suppose $\gamma\neq0$. If $\alpha=1$, then $\Tor_1^A(T_1,T_2)=0$ for every pair of 
one dimensional $A$-modules $T_1,T_2$, from Lemma \ref{lemma:cuentas_tor1}. Since $A\cong 
B$, this is also true for $B$ and one dimensional 
$B$-modules. By Lemma \ref{lemma:monomiales_maxtor1}, the quiver $Q$ has no arrows. This 
is impossible, and so $\alpha\neq1$. The same lemmas imply in this case that there is 
at most one arrow between each pair of vertices in $Q$. In particular, there is 
at most one element in $eQ_1e$ for every vertex $e$. 

Define $V=\{e\in Q_0: \#eQ_1e=1\}$ and for every $e\in V$, denote $a_e$ the unique 
element in $eQ_1e$.
Since $A$, and thus $B$, is of global dimension $3$, Bardzell's resolution \cite{Ba97}
of $B$ is of finite length. So ${a_e}^n\notin I$ for all $n\in\NN$. This implies 
$B_e=\KK[X]$ and $I_e=0$ for all $e\in V$, and $B_e=\KK$ for all $e\notin V$. By Lemma 
\ref{lemma:monomiales_techo},
\[
 \abel B \cong \bigoplus_{e\notin V}\KK\oplus\bigoplus_{e\in V} \KK[X].
\]
In particular, its group of units is contained in the finite dimensional vector space 
$\KK^{\#Q_0}$.
On the other hand, the fact that the ideals $\langle d,u\rangle$ and $\langle 
(1-\alpha)du-\gamma)\rangle$ in $\KK[d,u]$ are coprime implies
\[
 \abel A \cong \KK\oplus \frac{\KK[d,u]}{\langle 
(1-\alpha)du-\gamma\rangle}.
\]
The group of units of this algebra is contained in no finite dimensional space.
Since $\abel B$ is isomorphic to $\abel A$, this is a contradiction and 
we conclude the proof of Theorem \ref{teo:no_son_mono}.

\bibliographystyle{model1-num-names}

\begin{bibdiv}
\begin{biblist}
\bib{Ba97}{article}{
   author={Bardzell, Michael J.},
   title={The alternating syzygy behavior of monomial algebras},
   journal={J. Algebra},
   volume={188},
   date={1997},
   number={1},
   pages={69--89},
}
%
%
\bib{BR98}{article}{
   author={Benkart, Georgia},
   author={Roby, Tom},
   title={Down-up algebras},
   journal={J. Algebra},
   volume={209},
   date={1998},
   number={1},
   pages={305--344},
} 
%
%
%
\bib{CM00}{article}{
   author={Carvalho, Paula A. A. B.},
   author={Musson, Ian M.},
   title={Down-up algebras and their representation theory},
   journal={J. Algebra},
   volume={228},
   date={2000},
   number={1},
   pages={286--310},
}
%
%
\bib{CS15}{article}{
   author={Chouhy, Sergio},
   author={Solotar, Andrea},
   title={Projective resolutions of associative algebras and ambiguities},
   journal={J. Algebra},
   volume={432},
   date={2015},
   pages={22--61},
}

\bib{CHS16}{article}{
   author={Chouhy,Sergio},
   author={Herscovich, Estanislao},
   author={Solotar, Andrea},
   title={Hochschild homology and cohomology of down-up algebras},
   journal={arXiv:1609.09809},
   date={2016}
}

%
%
%
%
%

\bib{KK00}{article}{
   author={Kirkman, Ellen},
   author={Kuzmanovich, James},
   title={Non-Noetherian down-up algebras},
   journal={Comm. Algebra},
   volume={28},
   date={2000},
   number={11},
   pages={5255--5268},
}

\bib{KMP99}{article}{
   author={Kirkman, Ellen},
   author={Musson, Ian M.},
   author={Passman, D. S.},
   title={Noetherian down-up algebras},
   journal={Proc. Amer. Math. Soc.},
   volume={127},
   date={1999},
   number={11},
   pages={3161--3167},
}
%
%
\bib{RS06}{article}{
   author={Richard, Lionel},
   author={Solotar, Andrea},
   title={Isomorphisms between quantum generalized Weyl algebras},
   journal={J. Algebra Appl.},
   volume={5},
   date={2006},
   number={3},
   pages={271--285},
}
%
%
%

\bib{Sm12}{article}{
   author={Smith, S. Paul},
   title={``Degenerate'' 3-dimensional Sklyanin algebras are monomial
   algebras},
   journal={J. Algebra},
   volume={358},
   date={2012},
   pages={74--86},
}

\bib{SV15}{article}{
   author={Su{\'a}rez-Alvarez, Mariano},
   author={Vivas, Quimey},
   title={Automorphisms and isomorphisms of quantum generalized Weyl
   algebras},
   journal={J. Algebra},
   volume={424},
   date={2015},
   pages={540--552},
}
%
%
%
%
\bib{Zh99}{article}{
   author={Zhao, Kaiming},
   title={Centers of down-up algebras},
   journal={J. Algebra},
   volume={214},
   date={1999},
   pages={103--121},
}

\end{biblist}
\end{bibdiv}
\addcontentsline{toc}{section}{References}



\bigskip

\footnotesize
\noindent S.C.:
\\ IMAS, UBA-CONICET,
Consejo Nacional de Investigaciones Cient\'\i cas y T\'ecnicas, \\
Ciudad Universitaria, Pabell\'on I, 1428 Buenos Aires, Argentina
\\{\tt schouhy@dm.uba.ar}

\medskip

\noindent A.S.:
\\Departamento de Matem\'atica, Facultad de Ciencias Exactas y Naturales, Universidad
de Buenos Aires,\\
Ciudad Universitaria, Pabell\'on I, 1428 Buenos Aires, Argentina; and \\
IMAS, UBA-CONICET,
Consejo Nacional de Investigaciones Cient\'\i ficas y T\'ecnicas, Argentina
\\{\tt asolotar@dm.uba.ar}

\end{document}